\documentclass[12pt,reqno]{amsart}

\usepackage{amsmath}
\usepackage{amssymb}
\usepackage{amsthm}
\usepackage{latexsym}
\usepackage{color}

\newtheorem{theorem}{Theorem}[section]
\newtheorem{lemma}{Lemma}[section]
\newtheorem{corollary}{Corollary}[section]
\newtheorem{remark}{Remark}[section]
\newtheorem{definition}{Definition}[section]
\newtheorem{proposition}{Proposition}[section]
\newtheorem{example}{Example}[section]
\newtheorem{assumption}{Assumption}[section]
\numberwithin{equation}{section}

\newcommand{\bth}{\begin{theorem}}
\newcommand{\ethe}{\end{theorem}}

\newcommand{\bre}{\begin{remark}}
\newcommand{\ere}{\end{remark}}

\newcommand{\ble}{\begin{lemma}}
\newcommand{\ele}{\end{lemma}}
\newcommand{\bde}{\begin{definition}}
\newcommand{\ede}{\end{definition}}

\newcommand{\bco}{\begin{corollary}}
\newcommand{\eco}{\end{corollary}}

\newcommand{\bpr}{\begin{proposition}}
\newcommand{\epr}{\end{proposition}}

\newcommand{\bexer}{\begin{exercise}}
\newcommand{\eexer}{\end{exercise}}
\newcommand{\breh}{\begin{hint}}
\newcommand{\ereh}{\end{hint}}

\newcommand{\halmos}{\hfill \qed}

\newcommand{\bexam}{\begin{example}}
\newcommand{\eexam}{\end{example}}

\newcommand{\pr} {{\bf Proof.}}

\newcommand{\bfi}{\begin{fig}}
\newcommand{\efi}{\end{fig}}
\newcommand{\beao}{\begin{eqnarray*}}
\newcommand{\eeao}{\end{eqnarray*}\noindent}

\newcommand{\beam}{\begin{eqnarray}}
\newcommand{\eeam}{\end{eqnarray}\noindent}

\newcommand{\xto}{x\to\infty}

\newcommand{\bF}{\overline{F}}
\newcommand{\bV}{\overline{V}}

\newcommand{\bG}{\overline{G}}

\newcommand{\bbr}{{\mathbb R}}

\newcommand{\bbn}{{\mathbb N}}

\newcommand{\E}{\mathbf{E}}
\newcommand{\PP}{\mathbf{P}}

\newcommand{\vep}{\varepsilon}

\begin{document}
\title[Background risk model in presence of heavy tails under dependence]{Background risk model in presence of heavy tails under dependence}

\author[D.G. Konstantinides, C. D. Passalidis]{Dimitrios G. Konstantinides, Charalampos  D. Passalidis}

\address{Dept. of Statistics and Actuarial-Financial Mathematics,
University of the Aegean,
Karlovassi, GR-83 200 Samos, Greece}
\email{konstant@aegean.gr,\;sasm23002@sas.aegean.gr.}

\date{{\small \today}}

\begin{abstract}
In this paper, we examine two problems on applied probability, which are directly connected with the dependence in presence of  heavy tails. The first problem, is related to max-sum equivalence of the randomly weighted sums in bi-variate set up. Introducing a new dependence, called Generalized Tail Asymptotic Independence, we establish the bi-variate max-sum equivalence, under a rather general dependence structure, when the primary random variables follow distributions from the intersection of the dominatedly varying and the long tailed distributions. On base of this max-sum equivalence, we provide a result about the asymptotic behavior of two kinds of ruin probabilities, over a finite time horizon, in a bi-variate renewal risk model, with constant interest rate. The second problem, is related to the asymptotic behavior of the Tail Distortion Risk Measure, in a static portfolio, called Background Risk Model. In opposite to other approaches on this topic, we use a general enough assumption, that is based on multivariate regular variation.
\end{abstract}

\maketitle

\section{Introduction}

\subsection{Concepts and motivation.}

In last decades, the distributions with heavy tails play a crucial role in applied probability, and especially in risk theory and risk management, see \cite{chen:liu:2022}, \cite{leipus:siaulys:2024}, \cite{leipus:siaulys:konstantinides:2023}, \cite{li:2013}, \cite{li:tang:2015}, \cite{liu:yang:2021}, \cite{spindys:siaulys:2020}, etc.  

Simultaneously, the dependence modeling among risks, seems to have equally important impact on insurance applications, while keeps the mathematical interest with respect to generalizations either of some independent results, or of some counterexamples in which the independent results can not be generalized. Hence, we observe that the study of dependent models in combination with the presence of heavy tails, present a useful tool in applications and the same time a strong mathematical support.

In this paper we shall explore the concept to interdependence, in the sense of a complex dependence between two sequences of random variables, whose distributions are from the heavy tail class. In order to make it clear, we depict in two ways the interdependence, that are used in this paper and in the frame of the new results.

At first, we understand the interdependence, as structure between two finite sequences of primary random variables, heavy-tailed distributed, $X_1,\,\ldots,\,X_n$ and $Y_1,\,\ldots,\,Y_m$, such that each of them contains dependent components, but simultaneously the two sequences are also dependent each other. In this sense, we introduce a new dependence structure, called Generalized Tail Asymptotic Independent, symbollically $GTAI$, in the Definition \ref{def.KP.6} below, and it belongs to the family of second order asymptotic independence. Next, in section 2, we establish the max-sum equivalence of randomly weighted sums, in bi-variate framework, with random weights $\Theta_1,\,\ldots,\,\Theta_n$, $\Delta_1,\,\ldots,\,\Delta_m$, bounded from above, non-negative and non-degenerated to zero, that are arbitrarily dependent each other, and independent of $X_1,\,\ldots,\,X_n$ and $Y_1,\,\ldots,\,Y_m$, under dominatedly varying and long tailed distributions for the primary variables. Namely we establish the asymptotic relation
\beam \label{eq.KP.1.1}
\PP\left[\sum_{i=1}^{n}\Theta_i\,X_i>x\,, \;\sum_{j=1}^{m}\Delta_j\,Y_j>y\right]\sim \sum_{i=1}^{n}\sum_{j=1}^{m}\PP[\Theta_i\,X_i>x,\;\Delta_j\,Y_j>y]\,,
\eeam 
as $x\wedge y \rightarrow \infty$. The study of relation \eqref{eq.KP.1.1}, through this kind of interdependence covers a gap in the literature, and it includes many of the already existing results, as it is discussed in Section 2. Next, the establishment of the asymptotic equivalence in \eqref{eq.KP.1.1}, helps in Section 3, to provide asymptotic expressions for two types of  ruin probability, over finite time horizon, in the frame of bi-variate, continuous time risk model, with constant interest force, and a common renewal counting process for the two lines of business.

In Section 4, we meet a second type of interdependence. Since, in this section we focus on asymptotic behavior of a risk measure, called Tail Distortion Risk Measure, symbolically $TDRM$, with respect to a model, known as Backgroung Risk Model, direct role play the quantities of the form 
\beam  \label{eq.KP.1.2}
\sum_{i=1}^n \Theta_i\,X_i\,,
\eeam 
called randomly weighted sums. The interdependence now, is expressed by the dependence among the components of vectors  ${\bf \Theta}=(\Theta_1,\,\ldots,\,\Theta_n)$, and  ${\bf X}=(X_1,\,\ldots,\,X_n)$, and simultaneously between the vectors ${\bf \Theta}$ and ${\bf X}$. Such dependence structures were studied in the literature under the framework of multivariate regular variation for the distribution of ${\bf X}$. However, in larger classes of heavy tailed distributions for the components of ${\bf X}$, the interdependence effect appears rarely. The reason is, the difficulty to find max-sum equivalence for the randomly weighted sums, namely 
\beam \label{eq.KP.1.3}
\PP\left[\sum_{i=1}^{n}\Theta_i\,X_i>x\right]\sim \sum_{i=1}^{n} \PP[\Theta_i\,X_i>x]\,,
\eeam 
as $x\wedge y \rightarrow \infty$. We refer to \cite{chen:yuen:2009}, \cite{li:2013}, \cite{tang:yuan:2014}, \cite{Yang:Leipus:Siaulys:2012}, \cite{yang:wang:leipus:siaulys:2013} for papers that studied relation \eqref{eq.KP.1.3}, under a variety of dependence structures and several classes of heavy-tailed distributions. In the paper \cite{chen:cheng:2024}, was for first time established   relation \eqref{eq.KP.1.3}, through interdependence for the case $n=2$, in distribution classes larger than regular variation.

We remain in the frame of regular variation for the $X_1,\,\ldots,\,X_n$, but under the relaxed assumption that ${\bf \Theta}{\bf X}$ follows a multivariate regular varying distribution, we allow a wide spectrum of dependencies among the components of the sum in relation \eqref{eq.KP.1.2}. Furthermore, our assumption for the products is easily verified, when ${\bf X}$ follows multivariate regularly varying distribution, through several applications of Breiman's theorem on multivariate set up. 

Finally, in the last section, we study the asymptotic behavior of the $TDRM$, in a Backgroung Risk Model, which satisfies the general assumption that ${\bf \Theta}{\bf X}$ follows $MRV$.

\subsection{Notations.}

In this subsection, after some necessary notations, we introduce the preliminary material for the heavy-tailed distribution classes. Let denote ${\bf x}=(x_1,\,\ldots,\,x_n)$, the scalar product $c\,{\bf x}=(c\,x_1,\,\ldots,\,c\,x_n)$, $x\wedge y=\min\{x,\,y\}$, $x\vee y=\max\{x,\,y\}$, $x^+=x\vee 0$, $x-=(-x\vee 0)$, $\left\lfloor x \right\rfloor$ the integer part of $x$, and $e_{i}$ represents the vector, whose all the components are $0$, except the $i$-th, that is $1$. With ${\bf 1}_{\{A\}}$ we denote the indicator function on the set $A$. For two positive functions $f$ and $g$ we write $f(x)=O(g(x))$, as $x\to \infty$, if 
\beao 
\limsup_{x\rightarrow\infty}\frac{f(x)}{g(x)}<\infty\,.
\eeao 
and $f(x)=o(g(x))$, as $x\to \infty$, if 
\beao
	\lim_{x\rightarrow\infty}\frac{f(x)}{g(x)}=0\,,
\eeao
we write $f(x) \sim c\,g(x)$, as $\xto$, for some $c \in (0,\,\infty)$, if
\beao
\lim_{\xto} \dfrac{f(x)}{g(x)} =c\,.
\eeao 

We write $f(x)\asymp g(x)$, if it holds both $f(x)=O(g(x))$ and $g(x)=O(f(x))$ as $x\to \infty$. All the previous asymptotic notations hold for $x\wedge y \rightarrow \infty$, when we have positive, bi-variate functions. For example we write $f(x,\,y) \sim c\,g(x,\,y)$, as  $x\wedge y \rightarrow \infty$, with $c\in (0,\,\infty)$, if 
\beao
\lim_{x\wedge y \rightarrow\infty}\frac{f(x,\,y)}{g(x,\,y)}=c\,.
\eeao
Let denote by $V(x)=\PP(Z\leq x)$ the distribution of the random variable $Z$ and by 
$\bV(x):=1-V(x)=\PP(Z>x)$, its tail. 

Let see now the classes of heavy-tailed distributions and their properties. We assume that all the distributions have infinite right endpoint, that means $\bV(x)>0$ for all $x>0$. We say that distribution $V$ has heavy tail, and write $V\in \mathcal{H}$, if for any $\epsilon >0$ the relation
\beao
	\int_{- \infty}^{\infty}e^{\epsilon x}\,V(dx)= \infty\,,
\eeao
is true. 

We say that distribution $V$ has long tail (symbolically, $V\in \mathcal{L}$) if for some (or equivalently, for all) $a > 0$ holds 
\beao
	\lim_{x\rightarrow\infty}\dfrac{\bV(x-a)}{\bV(x)}=1\,.
\eeao
The class $\mathcal{L}$ represents a subclass of heavy-tailed distributions $\mathcal{L} \subset \mathcal{H}$. If $V\in \mathcal{L}$, then there exists some function $a:[0,\infty)\rightarrow [0,\infty) $ such that $a(x)\rightarrow\infty$, $a(x)=o(x)$, and $ \bV(x\pm a(x))\sim\bV(x)$, as $x\rightarrow\infty$. This $a(\cdot)$ is called insensitivity function. 

We say that distribution $V$ belongs to the class of subexponential distributions, and write $V\in\mathcal{S}$, if for some (or equivalently for any) $n=2,\,3,\,\ldots$ it holds 
\beao
\lim_{\xto} \dfrac{\bV^{n*}(x)}{\bV(x)} = n\,,
\eeao
where $V^{n*}$ is the n-fold convolution of distribution $V$ with itself. The classes $\mathcal{S},\,\mathcal{L},\,\mathcal{H}$ were introduced in \cite{chistyakov:1964}.

We say that $V$ has dominatedly varying tail, and write $V\in \mathcal{D}$, if holds
\beao
	\limsup_{x\rightarrow\infty}\dfrac{\bV(t\,x)}{\bV(x)}< \infty
\eeao 
for all (or equivalently, for some) $0<t<1$. Let us make clear, that $\mathcal{D}\subsetneq \mathcal{H}$ and $\mathcal{D}\not\subseteq \mathcal{S}$, $\mathcal{S}\not\subseteq \mathcal{D}$. However $\mathcal{D}\cap \mathcal{L} \equiv \mathcal{D}\cap \mathcal{S} \subset \mathcal{S}$.
Now, we remind some properties of regular variation. A random variable $Z$ with distribution $V$ is regularly varying with index $\alpha>0$ and write $V\in \mathcal{R}_{-\alpha}$ if it holds
\beao
\lim_{\xto}	\dfrac{\bV(tx)}{\bV(x)} = t^{-\alpha}\,,
\eeao
for any $t>0$. The class of distributions with regularly varying tails is contained in $\mathcal{D}\cap \mathcal{L}$ and further the following inclusion is true
\beao
\mathcal{R}:=\bigcup_{\alpha>0} \mathcal{R}_{-\alpha} \subset \mathcal{D}\cap \mathcal{L}\subset\mathcal{S}\subset\mathcal{L}\subset \mathcal{H}\,,
\eeao
see for example in \cite[Rem. 2.1, p. 21]{leipus:siaulys:konstantinides:2023}. Let us now consider the limits
\beao
\bV_{*}(t):= \liminf_{x\rightarrow\infty}\dfrac{\bV(tx)}{\bV(x)}\,, \qquad \bV^{*}(t):= \limsup_{x\rightarrow\infty}\dfrac{\bV(tx)}{\bV(x)}\,,
\eeao
for all $t>1$.

For some distribution $V$ the upper and lower Matuszewska indexes are  given by
\beao
J_{V}^+:= \inf\left\{-\dfrac{\log\bV_{*}(t)}{\log t} : t>1 \right\}\,,\qquad J_{V}^-:=\sup\left\{-\dfrac{\log\bV^{*}(t)}{\log t} : t>1\right\}\,.
\eeao
respectively. For these indexes, that appeared in \cite{matuszewska:1964}, the following relations hold. $V\in \mathcal{D}$ if and only if $0\leq J_{V}^-\leq J_{V}^+<\infty$ and if $V\in \mathcal{R}_{-\alpha}$, then $J_{V}^-=J_{V}^+=\alpha$, (see \cite{bingham:goldie:teugels:1987}, or \cite[Sec. 2.4]{leipus:siaulys:konstantinides:2023}).

The class of regularly varying distributions have many closure properties, see for example \cite{leipus:siaulys:konstantinides:2023}. One of this properties is the asymptotic behavior of the tail of product convolution, which is the popular Breiman's Theorem.  In the \cite{breiman}, and \cite{cline:samorodnitsky:1994} established the following result. 

If $Z$ and $\Theta$ two independent random variables with distribution of $Z$ from class $\mathcal{R}_{-a}$, for some $\alpha>0$, and $\Theta$ is non-negative, non-degenerated to zero, such that $\E[\Theta^{\alpha+\epsilon}]<\infty$ for some $\epsilon>0$ then 
\beao
\PP(\Theta\,Z>x)\sim \E(\Theta^{\alpha})\PP(Z>x)\,,
\eeao 
as $\xto$, which further means that the distribution of the product $\Theta\,Z$belongs to $\mathcal{R}_{-a}$. 

Now we can go to the extension of regular variation in random vectors. Let ${\bf X}$ be a random vector in the space $[0,\infty]^{n}$. We remind that  ${\bf X}$ follows the multivariate regularly varying distribution, symbolically $MRV$, if there exists a function $b\;:\;\mathbb{R}_+\rightarrow \mathbb{R}_+$ and a non-degenerated to zero, Radon measure $\mu$, such that for every $\mu$-continuous Borel, namely with $\mu(\partial B)=0$ where $\partial B$ represents the border of $B$, set  $B  \subseteq[0,\infty]^{n}\setminus \{{\bf 0}\}$, it holds
\beam \label{eq.KP.3}
\lim_{\xto} x\,\PP\left[\dfrac{{\bf X}}{b(x)}\in B \right] = \mu(B)\,,
\eeam
and we write ${\bf X}\in MRV(\alpha,\,b,\,\mu)$. This measure $\mu$ is homogeneous, namely for any Borel set $B \subseteq[0,\infty]^{n}\setminus\{ {\bf 0}\}$ we obtain 
\beao
\mu(tB)=t^{-\alpha}\mu(B)\,,
\eeao 
for any $t > 0$.

For the normalizing function $b(\cdot)$ we have that $b(\cdot)\in \mathcal{R}_{1/\alpha}$, as is indicated in \cite{resnick:2007}. Another representation of \eqref{eq.KP.3}  is in the following form 
\beao
\lim_{\xto} \dfrac{1}{\bV(x)}\PP\left[ \dfrac{{\bf X}}{x}\in B \right] = \mu(B)
\eeao
for a distribution $V \in \mathcal{R}_{-a}$. 

$MRV$ is a well-known multivariate distribution class with rich properties. We refer to \cite{resnick:2007} for several treatments and to \cite{chen:liu:2024}, \cite{cheng:konstantinides:wang:2022}, \cite{konstantinides:li:2016}, \cite{li:2016}, \cite{li:2022} for applications on risk theory and risk management.

It is worth to mention that recently there were some attempts to extend the heavy-tailed distributions to multivariate set up, see for example \cite{konstantinides:passalidis:2024b}, \cite{konstantinides:passalidis:2024g}, \cite{samorodnitsky:sun:2016} for such kind of approaches and survey of classes.

Now let us remind the strong asymptotic independence, that we need later, (see  \cite[Assumption A]{li:2018b}). We should notice that, the convergence in the next dependence structure was defined in the general case, as $(x,\,y) \rightarrow (\infty,\,\infty)$ in \cite{li:2018b}, but for sake of compactness of the text we use here the convergence $x \wedge y \rightarrow \infty$. Let $X$ and $Y$ two real-valued random variables with distributions $F$ and $G$ respectively. We say that $X$ and $Y$ are strongly asymptotically independent, symbolically $SAI$, if 
\beao
	\PP(X^{-}>x,\;Y>y)&=&O[F(-x)\,\bG(y)] \,,\quad  \PP(X>x,\;Y^{-}>y)=O[\bF(x)\,G(-y)]\,,\\[2mm]
	\PP(X>x,\;Y>y)&\sim& C\,\bF(x)\,\bG(y)\,,
\eeao 
as $x \wedge y \rightarrow \infty$, for some constant $C>0$.

\begin{remark} \label{rem.KP.2}
	It is easy to see that, SAI contains the independence as a special case.
	In the case where $X$ and $Y$ are non-negative (or generally, bounded from below) as in our case, then $X$ and $Y$ are SAI if holds
\beao
\PP(X>x,\;Y>y)\sim C\bF(x)\,\bG(y)
\eeao	
as $x \wedge y \rightarrow \infty$. The SAI covers a wide spectrum of dependence as for example Ali-Mikhail-Haq, Farlie-Gumbel-Morgenstern and Frank copulas, see \cite{li:2018b}. 
\end{remark}

\section{ Generalized tail asymptotic independence}

Next we establish the following relation
\beam \label{eq.KP.44}
	\PP(X_n(\Theta)>x,\;Y_m(\Delta)>y)\sim\sum_{i=1}^{n}\sum_{j=1}^{m}\PP(\Theta_i\,X_i>x,\;\Delta_j\,Y_j>y)\,,
\eeam
as $x\wedge y \rightarrow \infty$, where
\beao
	X_n(\Theta):=\sum_{i=1}^{n}\Theta_i\,X_i\,, \qquad Y_m(\Delta):=\sum_{j=1}^{m}\Delta_j\,Y_j\,.
\eeao
with $\{\Theta_i,\,\Delta_j,\; i,\,j \in \bbn\}$ arbitrarily dependent non-negative random variables, called random weights, and the primary random variables $ \{(X_i,\,Y_i),\; i\in \bbn\}$ are such that $X_i$ and $Y_i$ are SAI (but $X_i$ and $Y_j$ are independent for any $i\neq j$),  with $\PP(X_i>x)= \bF_i(x) \in \mathcal{D}\cap \mathcal{L}$ and $\PP(Y_j>x)= \bG_j(x) \in \mathcal{D}\cap \mathcal{L}$. 

The study of the joint behavior of the tail of two randomly weighted sums, provides a realistic framework for the insurance applications, since in most insurance companies are running several portfolios, which are subject to dependence environment, see for relation \eqref{eq.KP.44} under several heavy-tailed distribution classes and several dependence structures in \cite{chen:yang:2019}, \cite{li:2018}, \cite{shen:du:2023}, \cite{yang:chen:yuen:2024}. Indeed, we find mostly two forms of dependence structure. Firstly, the $\{\Theta_i,\,\Delta_j\}$ are arbitrarily dependent, while the $\{(X_i,\,Y_i) \}$ are independent random vectors, within each random vector appears some dependence structure. Secondly, the $\{\Theta_i,\,\Delta_j\}$ are arbitrarily dependent, and in each sequence $\{X_i \}$ and $\{Y_i \}$ appears some dependence structure, while the two sequences $\{X_i \}$ and $\{Y_i \}$ are independent. In this paper we combine these two approaches, through the following definition. In next definition, we use the random variables  $X_{1},\,\ldots,\,X_{n}$ and $Y_{1},\,\ldots,\,Y_{m}$, that follow distributions with supports, which are not bounded from above.

\begin{definition} \label{def.KP.6}
	Let $X_{1},\,\ldots,\,X_{n}$ and $Y_{1},\,\ldots,\,Y_{m}$ real valued random variables. Then we say that $X_{1},\,\ldots,\,X_{n},\,Y_{1},\,\ldots,\,Y_{m}$ are generalized tail asymptotic independent, symbolically GTAI, if both following relations hold 
	\beam \label{eq.KP.45}
		\lim_{{x_{i}\wedge x_{k} \wedge y_{j}}\rightarrow\infty}\PP(|X_{i}|>x_{i}\; \mid \;X_{k}>x_{k},\; Y_{j}>y_{j})=0\,,
	\eeam
	for all $1\leq k\neq i \leq n$, $j=1,\,\ldots,\,m$ and 
	\beam \label{eq.KP.46}
		\lim_{{x_{i}\wedge y_{j} \wedge y_{k}}\rightarrow\infty}\PP(|Y_{j}|>y_{j} \;\mid\; X_{i}>x_{i},\; Y_{k}>y_{k})=0\,,
	\eeam
	for all $1\leq j \neq k \leq m$, $i=1,\,\ldots,\,n$.
\end{definition}

\begin{remark} \label{rem.KP.5} 
	This dependence structure allows dependence between $X_{1},\,\ldots,\,X_{n}$, between $Y_{1},\,\ldots,\,Y_{m}$ and dependence between $X_{i}$ and $Y_{j}$ (not only for $i=j$). In this paper we restrict ourselves in the case with $X_{i},\,Y_{i}$ to be SAI dependent for the same $i$, and $X_{i},\, Y_{j}$ independent for any $i\neq j$. We have to note that the $GTAI$ structure belongs to the family of 'second order asymptotic independnce', that means that the probability of three or more extremal events, is negligible with respect to the probability of two extremal events, namely one in each sequence.

Notice that if $X_{i}$ and $Y_{j}$ are independent  for any $i,j \in \bbn$ (i.e. the two sequences are independent) then the relationships \eqref{eq.KP.45}, \eqref{eq.KP.46}  can be written as follows
\beam \label{eq.KP.47}
	\lim_{x_i\wedge x_k\rightarrow\infty}\PP(|X_i|>x_i \;\mid\; X_k>x_k)=0
\eeam
holds for all $1\leq i\neq k\leq n$ and 
\beam \label{eq.KP.48}
	\lim_{y_j\wedge y_k\rightarrow\infty}\PP(|Y_j|>y_j \;\mid\; Y_k>y_k)=0
\eeam
for all $1\leq j\neq k\leq m$. Through \eqref{eq.KP.47}, \eqref{eq.KP.48} we obtain the definition of tail asymptotic independence of $X_{1},\,\ldots,\,X_{n}$ and $Y_{1},\,\ldots,\,Y_{m}$ respectively, introduced in \cite{geluk:tang:2009}. We wonder if the results of our paper can be identified using instead of $GTAI$ the $TAI$ over the $X_1,\,\ldots,\,X_n,\,Y_1\,\ldots,\,Y_m$, as we find under similar frame in \cite{chen:wang:wang:2013}. The reply is no, because, inspite of the presence of 'interdependence' in both cases, the $GTAI$ studies second order asymptotic independence events, while the $TAI$ studies only first order asymptotic independence events. Hence, in Theorem \ref{th.KP.4} below, we demonstrate the 'insensitivity' with respect to dependence in a more extremal event, in comparison with $TAI$ case.
\end{remark}

As follows from the last remark, if we choose two mutually independent sequences, where each one has tail asymptotic independent terms, then the structure $GTAI$ is satisfied. We present now two examples, which contains interdependence among the two sequences and also satisfies the $GTAI$ structure. For sake of simplicity, we restrict ourselves on non-negative random variables, with $n=m=2$.

\bexam \label{exam.KP.2.1}
Let $X_1,\,X_2,\,Y_1,\,Y_2$ be non-negative random variables and $Z_1,\,Z_2,\,Z_3,\,Z_4$, with $Z_i \in \{X_1,\,X_2,\,Y_1,\,Y_2\}$, where $Z_i\neq Z_j$, for any $1 \leq i \neq j \leq 4$. Let $z_1,\,z_2,\,z_3,\,z_4$, with  $z_i \in \{x_1,\,x_2,\,y_1,\,y_2\}$, and let permit $z_i=z_j$, for any $1 \leq i \neq j \leq 4$. We assume that the $Z_1,\,Z_2,\,Z_3,\,Z_4$ are widely upper orthant dependent, see \cite{wang:wang:gao:2013}, namely for any integer $n=1,\,\ldots,\,4$, there exists a positive number $g_u(n)$, such that for any $z_i \in \bbr$, with $i=1,\,\ldots,\,n$ it holds
\beao
\PP\left[\bigcap_{i=1}^n \{Z_i > z_i\} \right] \leq g_u(n)\,\prod_{i=1}^n \PP[Z_i>z_i]\,.
\eeao
From this relation with $n=3$, follows directly the $GTAI$ structure.
\eexam

Except the advantage to imply the $GTAI$ structure, in next example we get an idea about the dependence frames that satisfy the conditions of Theorem \ref{th.KP.4} below.

\bexam \label{exam.KP.2.2}
Under the notation of Example \ref{exam.KP.2.1}, we consider that the $Z_1,\,Z_2,\,Z_3,\,Z_4$ are $SAI$, for any two of them, namely, say for any $Z_i,\,Z_j$, where $i\neq j$, there exists a constant $C_{ij}>0$, such that $\PP[Z_i>z_i\,,\;Z_j>z_j] \sim C_{ij}\PP[Z_i>z_i]\,\PP[Z_j>z_j]$, as $z_i\wedge z_j \to \infty$, and further we consider that the are  $SAI$, for any three of them, namely, say for any $Z_i,\,Z_j,\,Z_k$, where $i\neq j\neq k$, there exists a constant $C_{ijk}>0$, such that 
\beao
\PP[Z_i>z_i\,,\;Z_j>z_j\,,\;Z_k>z_k] \sim C_{ijk}\PP[Z_i>z_i]\,\PP[Z_j>z_j]\,\PP[Z_k>z_k]\,,
\eeao 
as $z_i\wedge z_j\wedge z_k \to \infty$. From the triple $SAI$ structure, is directly implied the $GTAI$.
\eexam

Now we can give the first assumption for the main result of the section.
\begin{assumption} \label{ass.KP.2}
	We assume that the following random variables $X_{1},\,\ldots,\,X_{n},\,Y_{1},\,\ldots,\,Y_{m}$, are $GTAI$, and the random weights $\Theta_{1}\,,\ldots,\,\Theta_{n},\,\Delta_{1},\,\ldots,\,\Delta_{m}$, are non-negative and non-degenerated to zero random variables,  that follow distributions, whose supports are bounded from above, and are independent of  $X_{1},\,\ldots,X_{n},\,Y_{1},\,\ldots,\,Y_{m}$.
\end{assumption}

In next lemma, we see that under Assumption \ref{ass.KP.2} the $GTAI$ structure, remains invariant with respect to products.
\begin{lemma} \label{lem.KP.2.1}
Under the Assumption \ref{ass.KP.2}, we obtain that all the products $\Theta_{1}X_{1},\,\ldots,\,\Theta_{n}X_{n}$, $\Delta_{1}Y_{1},\,\ldots,\,\Delta_{m}Y_{m}$, are $GTAI$.
\end{lemma}
\begin{proof}
By definition of $GTAI$, for any $\epsilon>0$, there exist some constant $\kappa_{0}>0$ such that for any $x_{i}\wedge x_{k}\wedge y_{j}>\kappa_{0}$, it holds the relation
$\PP(|X_{i}|>x_{i}\mid X_{k}>x_{k}, Y_{j}>y_{j})<\epsilon$, for all $1\leq i \neq k \leq n$, $j=1,\,\ldots,\,m$. 	
Let $\PP(\Theta_{i}\leq b_{i})=1$, $\PP(\Delta_{j}\leq d_{j})=1$, where $b_{i},\,d_{j}>0$, for all $i=1,...,n$ and $j=1,...,m$. Then for $x_{i},x_{k},y_{j}$ sufficiently large, namely 
\beao
\dfrac{x_{i}}{b_{i}}>\kappa_{0}\,,\quad \dfrac{x_{k}}{b_{k}}>\kappa_{0}\,, \quad \dfrac{y_{j}}{\lambda_{j}}>\kappa_{0}\,,
\eeao
we have that:
\beao
&&\PP(|\Theta_{i} X_{i}|>x_{i}, \Theta_{k}X_{k}>x_{k},\Delta_{j}Y_{j}>y_{j})\\[2mm]
&&=\int_{0}^{b_{i}}\int_{0}^{b_{k}}\int_{0}^{d_{j}}\PP\left[|X_{i}|>\frac{x_{i}}{c_{i}}, X_{k}>\frac{x_{k}}{c_{k}}, Y_{j}>\frac{y_{j}}{\lambda_{j}}\right]\PP(\Theta_{i}\in dc_{i}, \Theta_{k}\in dc_{k}, \Delta_{j}\in d\lambda_{j})\\[2mm]
&&=\int_{0}^{b_{i}}\int_{0}^{b_{k}}\int_{0}^{d_{j}}\PP\left[|X_{i}|>\frac{x_{i}}{c_{i}}\mid X_{k}>\frac{x_{k}}{c_{k}}, Y_{j}>\frac{y_{j}}{\lambda_{j}}\right]\PP\left[X_{k}>\frac{x_{k}}{c_{k}}, Y_{j}>\frac{y_{j}}{\lambda_{j}}\right]\\[2mm]
&& \qquad \times \PP(\Theta_{i}\in dc_{i}, \Theta_{k}\in dc_{k}, \Delta_{j}\in d\lambda_{j})\leq\epsilon \PP(\Theta_{k}X_{k}>x_{k}, \Delta_{j}Y_{j}>y_{j})
\eeao
by arbitrariness of $\epsilon>0$, we have the first relation of $GTAI$. The symmetrical relation, namely $	\PP(|\Delta_{j}Y_{j}|>y_{j}, \Theta_{i}X_{i}>x_{i}, \Delta_{k}Y_{k}>y_{k})\leq\epsilon \PP(\Theta_{i}X_{i}>x_{i}, \Delta_{k}Y_{k}>y_{k})$, for any $1\leq j \neq k \leq m$, $i=1,\,\ldots,\,n$, can be easily obtained through similar arguments. And this complete the proof.
\end{proof}	

\begin{assumption} \label{ass.KP.3}
	Let $ \{(X_i,\,Y_i),\;i\in \bbn \}$ be some sequence of random vectors with marginal distributions $F_i \in \mathcal{D}\cap \mathcal{L}$ and $G_i\in \mathcal{D}\cap \mathcal{L}$, respectively, for all $i\in \bbn$. Assume that $X_i$ and $Y_i$ are SAI for the same $i$, with constant $C_i>0$, and $X_{i},\,Y_{j}$ are independent for any $i\neq j$. 
\end{assumption}

The next lemma plays crucial role in the proof of Theorem \ref{th.KP.4}.
\begin{lemma} \label{lem.KP.1} 
	Under Assumption \ref{ass.KP.2} and Assumption  \ref{ass.KP.3}, we find
	\beao 
		\PP(\Theta_i\,X_i>x,\; \Delta_j\,Y_j>y,\; \Theta_k\,|X_k|>a(x))=o\left(\PP(\Theta_{i}\,X_{i}>x,\;\Delta_{j}\,Y_{j}>y)\right)\,,
	\eeao
as $x\wedge y \rightarrow \infty$, for any $1\leq i \neq k \leq n$ with $j=1,\,\ldots,\,m$, where $a(x)>0$, is such that $a(x) \rightarrow \infty$, $a(x)=o(x)$.
\end{lemma}

\begin{proof}
	We can see that
	\beao
		&&\PP(\Theta_i\,X_i>x,\; \Delta_j\,Y_j>y,\; \Theta_k\,|X_k|>a(x))\\[2mm]
		&&=\PP(\Theta_{k}\,|X_{k}|>a(x)\;\mid\; \Theta_{i}\,X_{i}>x,\;\Delta_{j}\,Y_{j}>y)\,\PP(\Theta_{i}\,X_{i}>x,\;\Delta_{j}\,Y_{j}>y)\\[2mm]
		&&=o\left[ \PP(\Theta_{i}\,X_{i}>x,\; \Delta_{j}\,Y_{j}>y)\right]\,,
	\eeao
as $x\wedge y \rightarrow \infty$, where in the last step we used Assumption \ref{ass.KP.2} and Lemma \ref{lem.KP.2.1}.
\end{proof}

Now, we can present the first main result.
\begin{theorem} \label{th.KP.4}
	Under Assumption \ref{ass.KP.2} and Assumption \ref{ass.KP.3}, for every pair $(n,\,m)\in \bbn^{2}$ we obtain \eqref{eq.KP.44}, as $x\wedge y \rightarrow \infty$.
\end{theorem}

\begin{proof}
Let us follow the line of \cite[Th. 1]{li:2018}. We consider the following events 
	\beao
		A_{x}^{\pm}:=\left\{ \bigvee_{i=1}^{n}\Theta_{i}X_{i}>x\pm a(x) \right\}\,, \qquad A_{y}^{\pm}:=\left\{ \bigvee_{j=1}^{m}\Delta_{j}Y_{j}>y\pm a(y) \right\}\,.
	\eeao
where $a(x)>0$, is such that $a(x) \rightarrow \infty$, $a(x)=o(x)$ and $a \in \mathcal{R}_{0}$. Now, let us define the probabilities 
	\beao
		I_{1}(x,\,y) &:=& \PP(X_n(\Theta)>x,\;Y_m(\Delta)>y,\;A_{x}^{-},\;A_{y}^{-})\,, \\[2mm]
		I_{2}(x,\,y) &:=& \PP(X_n(\Theta)>x,\;Y_m(\Delta)>y,\;(A_{x}^{-})^{c})\,, \\[2mm]
		I_{3}(x,\,y) &:=& \PP(X_n(\Theta)>x,\;Y_m(\Delta)>y,\;(A_{y}^{-})^{c})\,,
	\eeao
Hence, we can see that
\beam \label{eq.KP.55}
\PP(X_n(\Theta)>x,\;Y_m(\Delta)>y)\leq I_{1}(x,\,y)+I_{2}(x,\,y)+I_{3}(x,\,y)\,.
\eeam
Therefore, for the upper bound of the probability in the left hand side of \eqref{eq.KP.55}, it remains to estimate the $I_{1}(x,\,y),\,I_{2}(x,\,y),\,I_{3}(x,\,y)$.	
\beao
I_{1}(x,\,y) &\leq& \PP(A_{x}^{-},\,A_{y}^{-})=\PP\left[\bigvee_{i=1}^{n}\Theta_{i}\,X_{i}>x-a(x),\; \bigvee_{j=1}^{m}\Delta_{j}\,Y_{j}>y-a(y)\right]\\[2mm]
&\leq& \sum_{i=1}^{n}\sum_{j=1}^{m}\PP\left[\Theta_{i}\,X_{i}>x-a(x), \;\Delta_{j}\,Y_{j}>y-a(y)\right]\\[2mm]
&\sim& \sum_{i=1}^{n}\sum_{j=1}^{m}\PP(\Theta_{i}X_{i}>x,\,\Delta_{j}Y_{j}>y)\,,
\eeao
as $x\wedge y \rightarrow \infty$, where in the last step we used \cite[Lemma 3(ii)]{li:2018} (as far as, $X_i$ and $Y_i$ are SAI). Next
	\beao
		&&I_{2}(x,\,y) =\PP\left[X_n(\Theta)>x,\; Y_m(\Delta)>y,\;\bigvee_{i=1}^{n}\Theta_{i}\,X_{i}>\dfrac{x}{n},\;\bigvee_{j=1}^{m}\Delta_{j}\,Y_{j}>\dfrac{y}{m},\;(A_{x}^{-})^{c}\right] \\[2mm]  
		&&=\PP\left[X_n(\Theta)>x, \;Y_m(\Delta)>y,\;\bigvee_{i=1}^{n}\Theta_{i}\,X_{i}>\dfrac{x}{n},\;\bigvee_{j=1}^{m}\Delta_{j}\,Y_{j}>\dfrac{y}{m},\;\bigvee_{k=1}^{n}\Theta_{k}\,X_{k} \leq x-a(x) \right] \\[2mm]
		&&\leq\sum_{i=1}^{n}\sum_{j=1}^{m}\sum_{1=k \neq i}^{n}\PP\left[\Theta_{i}\,X_{i}>\dfrac{x}{n},\;\Delta_{j}\,Y_{j}>\dfrac{y}{m}, \;\Theta_{k}\,X_{k}>\dfrac{a(x)}{n}\right] \\[2mm] 
		&&=o\left(\sum_{i=1}^{n}\sum_{j=1}^{m}\PP(\Theta_{i}\,X_{i}>x,\;\Delta_{j}\,Y_{j}>y) \right)\,,
	\eeao
as $x\wedge y \rightarrow \infty$, where in the last step we used Lemma \ref{lem.KP.1}  and \cite[Lemma 3(i)]{li:2018}, due to $F,\,G \in \mathcal{D}\cap \mathcal{L} \subsetneq \mathcal{D}$.
	
Symmetrically, we find 
\beao
I_{3}(x,\,y) =o\left(\sum_{i=1}^{n}\sum_{j=1}^{m}\PP(\Theta_{i}\,X_{i}>x,\;\Delta_{j}\,Y_{j}>y) \right)\,,
\eeao
as $x\wedge y \rightarrow \infty$. So from \eqref{eq.KP.55} we obtain
\beao
\PP(X_n(\Theta)>x,\;Y_m(\Delta)>y)\lesssim \sum_{i=1}^{n}\sum_{j=1}^{m}\PP(\Theta_{i}\,X_{i}>x,\;\Delta_{j}\,Y_{j}>y)\,,
\eeao
as $x\wedge y \rightarrow \infty$. 
	
For the lower bound of $\PP(X_n(\Theta)>x,\;Y_m(\Delta)>y)$, we get the following inequality $		\PP(X_n(\Theta)>x,\; Y_m(\Delta)>y)\geq \PP(X_n(\Theta)>x, \;Y_m(\Delta)>y, \;A_{x}^{+}, \;A_{y}^{+})$. Applying Bonferroni inequality twice we obtain
\beam \label{eq.KP.59} 
&&\PP(X_n(\Theta)>x, \;Y_m(\Delta)>y, \;A_{x}^{+}, \;A_{y}^{+})\\[2mm] \notag
&&\geq \sum_{i=1}^{n}\sum_{j=1}^{m}\PP(X_n(\Theta)>x, \;Y_m(\Delta)>y, \;\Theta_{i}X_{i}>x+a(x), \;\Delta_{j}Y_{j}>y+a(y))  \\[2mm] \notag
&&\quad-\sum_{1\leq i<k\leq n}\sum_{j=1}^{m}\PP(\Theta_{i}X_{i}>x+a(x), \;\Theta_{k}X_{k}>x+a(x), \;\Delta_{j}Y_{j}>y+a(y))	\\[2mm] \notag
&&\quad-\sum_{i=1}^{n}\sum_{1\leq j<k\leq m}\PP(\Theta_{i}X_{i}>x+a(x),\;\Delta_{j}Y_{j}>y+a(y),\;\Delta_{k}Y_{k}>y+a(y))\,,
\eeam
and further, by \cite[Lem. 3 (ii)]{li:2018} and by Lemma \ref{lem.KP.2.1}, the last two terms in \eqref{eq.KP.59} are asymptotically negligible with respect to $\sum_{i=1}^{n}\sum_{j=1}^{m} \PP(\Theta_{i}\,X_{i}>x,\; \Delta_{j}\,Y_{j}>y)$, as $x\wedge y \rightarrow \infty$. For the first term of right hand side in \eqref{eq.KP.59} we find a lower bound 
\beao
 &&\sum_{i=1}^{n}\sum_{j=1}^{m} \PP(\Theta_{i}X_{i}>x+a(x), \;\Delta_{j}Y_{j}>y+a(y))
		\\[2mm]
		&&\quad -\sum_{i=1}^{n}\sum_{j=1}^{m}\,\PP\left[\Theta_{i}X_{i}>x+a(x),\;\Delta_{j}Y_{j}>y+a(y),\,\sum_{k=1}^n \Theta_{k}\,X_{k} \leq x \right] \\[2mm]
		&&\quad -\sum_{i=1}^{n}\sum_{j=1}^{m}\,\PP\left[\Theta_{i}X_{i}>x+a(x),\;\Delta_{j}Y_{j}>y+a(y),\;\sum_{k=1}^m\Delta_{k}Y_{k} \leq y \right]\\[2mm]
		&&\geq \sum_{i=1}^{n}\sum_{j=1}^{m}\PP(\Theta_{i}X_{i}>x+a(x), \;\Delta_{j}Y_{j}>y+a(y))
		\\[2mm]
		&&\quad -\sum_{i=1}^{n}\sum_{j=1}^{m}\sum_{ i\neq k=1 }^{n}\PP\left[\Theta_{i}X_{i}>x+a(x), \;\Delta_{j}Y_{j}>y+a(y), \;\Theta_{k}X_{k}<-\dfrac{a(x)}{n}\right] \\[2mm]
		&&\quad -\sum_{i=1}^{n}\sum_{j=1}^{m}\sum_{ j\neq k=1 }^{m}\PP\left[\Theta_{i}X_{i}>x+a(x), \;\Delta_{j}Y_{j}>y+a(y), \;\Delta_{k}Y_{k}<-\dfrac{a(y)}{m}\right]\,,
\eeao
which by Lemma \ref{lem.KP.1} and \cite[Lem. 3(ii)]{li:2018} is asymptotically greater than the double sum $\sum_{i=1}^{n}\sum_{j=1}^{m}\PP(\Theta_{i}X_{i}>x, \;\Delta_{j}Y_{j}>y)$, as $x\wedge y \rightarrow \infty$, and thus we have the asymptotic relation \eqref{eq.KP.44}.
\end{proof}

\begin{remark} \label{rem.KP.5,5}
	We proved asymptotic estimation \eqref{eq.KP.44}, when $X_1,\,\ldots,\,X_n,\,Y_1,\,\ldots,\,Y_m$ are GTAI, as it comes by Definition \ref{def.KP.6}, the $X_i$, $Y_i$ are SAI and the $\Theta_i,\; \Delta_i$ are upper bounded and non-negative  random variables. This restriction is relatively small compared to the extension made in terms of dependence and at the same time reasonable in these models, since $\Theta_i$ and $\Delta_j$ depicts the discount factors.
\end{remark}

Let us remind that, $C_i> 0$ is the constant from $SAI$ condition, on each pair $(X_i,\,Y_i)\,,\;i=1,\,\ldots,\,n\wedge m$, see Assumption \ref{ass.KP.3}. Now applying \cite[Lemma 3 (iii)]{li:2018} in Theorem \ref{th.KP.4} and we have the following consequence.

\begin{corollary} \label{cor.KP.6} 
Under the conditions of Theorem \ref{th.KP.4}, with the restriction that $F_{i}\in \mathcal{R}_{-\alpha_{i}}$ and $G_{j}\in \mathcal{R}_{-\alpha'_{j}}$, with $\alpha_{i},\;\alpha'_{j}\in[0,\,\infty)$,  for any $i \in \bbn$, $j\in \bbn$, then we obtain 	
\beam \label{eq.KP.31}
&&\PP(X_n(\Theta)>x, \;Y_m(\Delta)>y)\\[2mm] \notag
&&\qquad \sim\sum_{i=1}^{n}\sum_{1=j \neq i}^{m}\E\left[\Theta_i^{\alpha_i}\,\Delta_j^{\alpha'_j}\right]\bF_i(x)\,\bG_j(y)+\sum_{i=1}^{n\wedge m}C_i\,\E\left[\Theta_i^{\alpha_i}\,\Delta_i^{\alpha'_i}\right] \bF_i(x)\,\bG_i(y)\,,
\eeam
as $x\wedge y \rightarrow \infty$.	
\end{corollary}

\begin{remark}\label{rem.KP.7}
We write $X^{\pm}_{n}(\Theta):=\sum_{i=1}^{n}\Theta_{i}X^{\pm}_{i}$, $Y^{\pm}_{m}(\Delta):=\sum_{j=1}^{m}\Delta_{j}Y^{\pm}_{j}$.
	We see that, the corresponding maximums for $X^{\pm}_{n}(\Theta)$ and $Y^{\pm}_{m}(\Delta)$, namely $\bigvee_{i=1}^{n}X_{i}(\Theta)$ and  $\bigvee_{j=1}^{m}Y_{j}(\Delta)$, satisfy the inequalities
	\beao
		\PP(X_{n}(\Theta)>x,\,Y_{m}(\Delta)>y)&\leq& \PP\left[\bigvee_{i=1}^{n}X_{i}(\Theta)>x,\;\bigvee_{j=1}^{m}Y_{j}(\Delta)>y\right]\\[2mm]
		&\leq& \PP(X^{+}_{n}(\Theta)>x,\;Y^{+}_{m}(\Delta)>y)\,,
	\eeao
Therefore Theorem \ref{th.KP.4} and Corollary \ref{cor.KP.6} are also satisfied, namely relations \eqref{eq.KP.44} and \eqref{eq.KP.31} after replacement of the pair $(X_n(\Theta),\,Y_m(\Delta))$, with $		\left(\bigvee_{i=1}^{n}X_{i}(\Theta),\;\bigvee_{j=1}^{m}Y_{j}(\Delta)\right)$. 
\end{remark}

Finally, we present one bi-dimensional discrete time risk model. Recently, the discrete time one-dimensional models have attracted attention by many researchers, see \cite{li:tang:2015}, \cite{yang:konstantinides:2015}, \cite{Yang:Leipus:Siaulys:2012(b)}. On the other hand more researchers study the multivariate risk models, because it is rarely for an insurance company to operate with one line of business, see \cite{cheng:yu:2019}, \cite{cheng:yu:2020b}, \cite{cheng:konstantinides:wang:2024b}, \cite{konstantinides:li:2016} among others.

We limit us in only two portfolios and discrete time, where $X_{i}$ and $Y_{i}$ depicts the net loss in $i$-th period, in the first and second line of business respectively, $\Theta_{i}$ and $\Delta_{i}$ they continue to be discount factors, of the $i$-th period.

Therefore, the stochastic surplus process of insurer at time $n\in \bbn$ is described by 
\beao
	\textbf{S}_{n}:=(S_{1n},\;S_{2n})=\left(x-\sum_{i=1}^{n}\Theta_{i}\,X_{i},\;y-\sum_{j=1}^{n}\Delta_{j}\,Y_{j}\right)\,,
\eeao 
where $x,y$ are initial capitals in each line of business, hence  one type of ruin probability is given as
\beao
	\psi(x,y,n):=\PP\left[\bigvee_{i=1}^{n}X_{i}(\Theta)>x,\;\bigvee_{j=1}^{n}Y_{j}(\Delta)>y\right]\,,
\eeao
for any $n\in \bbn$. This ruin probability depicts the probability that both portfolios have been with negative surplus during the $n$ first periods, but not necessarily simultaneously.

\begin{corollary}
	(i) Under the conditions of Theorem \ref{th.KP.4}, we obtain
\beao
\psi(x,y,n)\sim\sum_{i=1}^{n}\sum_{j=1}^{n}\PP(\Theta_{i}\,X_{i}>x,\;\Delta_{j}\,Y_{j}>y)\,,
\eeao
as $x\wedge y \rightarrow \infty$.
	
	(ii) Under the conditions of Corollary \ref{cor.KP.6} we obtain
\beao
\psi(x,y,n)&\sim& \PP(X_n(\Theta)>x, \;Y_n(\Delta)>y) \\[2mm] \notag
&\sim& \sum_{i=1}^{n}\sum_{1=j \neq i}^{n}\E\left[\Theta_{i}^{\alpha_{i}}\Delta_{j}^{\alpha'_{j}}\right]\bF_{i}(x)\bG_{j}(y)+\sum_{i=1}^{n}C_{i}\E\left[\Theta_{i}^{\alpha_{i}}\Delta_{i}^{\alpha'_{i}}\right]\,\bF_{i}(x)\,\bG_{i}(y)\,,
\eeao
as $x\wedge y \rightarrow \infty$.
\end{corollary}

\begin{proof}
	Directly from Remark \ref{rem.KP.7}, the definition of ruin probability, Theorem \ref{th.KP.4} and Corollary \ref{cor.KP.6}, respectively.
\end{proof}

\section{Application on bi-dimensional renewal risk model} \label{sec.KP.3}

Recently, the bi-dimensional risk model gained more popularity. The reason is the improvement in modeling practical insurance issues, while in the same time provide a relatively flexible mathematical frame, in comparison with the multidimensional ones, see \cite{chen:wang:wang:2013}, \cite{jiang:wang:chen:xu:2015}, \cite{li:2018b}, \cite{sun:geng:wang:2021}, \cite{sun:yuan:lu:2024}, \cite{yang:li:2014}, etc. In all these papers, were studied several models, not necessarily renewal ones, and several dependence structures, either among claims of the two business lines, or between  times and claims, or both. However, we can not find interdependence among the two lines of business. Inspired by the simple and concise risk model, found in \cite{yang:li:2014} and \cite{li:2018b}, we attempt a partial extension using Theorem \ref{th.KP.4}.

The discounted bi-variate surplus process $(U_1(t),\,U_2(t))^T$, for $t\geq 0$ has the form
\beam \label{eq.KP.3.1}
\left( 
\begin{array}{c}
U_{1}(t) \\[2mm]
U_{2}(t) 
\end{array} 
\right) =\left( 
\begin{array}{c}
x \\[2mm] 
y 
\end{array} 
\right) +\left( 
\begin{array}{c}
\int_{0-}^{t}e^{-r\,s}\,C_1(ds) \\[3mm] 
\int_{0-}^{t}e^{-r\,s}\,C_2(ds) 
\end{array} 
\right) -\left( 
\begin{array}{c}
D^{(1)}_r(t) \\[2mm] 
D^{(2)}_r(t)  
\end{array} 
\right)\,,
\eeam
where $(x,\,y)$ is the vector of the initial capitals for the two lines of business, $r\geq 0$, is the constant interest rate, $D^{(1)}_r(t) := \sum_{i=1}^{N(t)} X_i\,e^{-r\,T_i}$, $D^{(2)}_r(t) := \sum_{j=1}^{N(t)} Y_j\,e^{-r\,T_j}$, are the discounted aggregate claims of each line, up to time $t\geq  0$, the $\{(C_1(t),\,C_2(t))\,,\; t\geq 0\}$ is the premiums accumulation process for the two business lines, which represent non-decreasing cadlag paths with $(C_1(0),\,C_2(0))=(0,\,0)$, and the $\{(X_i,\,Y_i)\,,\;i\in \bbn\}$ is the sequence of claims, which arrive at the time moments $\{T_i\,,\;i\in \bbn\}$, that represents a renewal counting process
\beao
N(t) = \sum_{i=1}^{\infty} {\bf 1}_{\{\tau_i \leq t\}}\,,
\eeao 
for any $t\geq 0$, with renewal mean 
\beao
\lambda(t) = \E[N(t)] = \sum_{i=i}^{\infty} \PP[T_i \leq t]\,.
\eeao 
Indeed, the $\{ N(t)\,,\;t\geq 0\}$ represents a homogeneous, renewal process, namely the $\{\theta_i\,,\;n \in \bbn\}$, with $\theta_1 = T_1$, $\theta_i=T_i-T_{i-1}$, for integer $i \geq 2$, representing the inter-arrival times between two successive arrival times, is a sequence of independent and identically distributed, positive random variables.

In risk model \eqref{eq.KP.3.1}, we can examine several kinds of ruin probability, over finite time horizon of length $T>0$, such that it satisfies $\lambda(T) > 0$. Let us define the ruin probability as follows
\beam \label{eq.KP.3.3}
\psi_*(x,\,y;\,T) := \PP[T_* \leq T\;|\;(U_1(0),\,U_2(0))=(x,\,y)]\,,
\eeam
where $*$ is either '$\max$' or '$\and$', and as $T_*$ we consider the 
\beam \label{eq.KP.3.4}
T_{\max}&:=& \inf\left\{t>0\;:\;\left[U_1(t) \vee  U_2(t)\right] < 0\right\}\,, \\[2mm] \notag
T_{\and}&:=& \inf\left\{t>0\;:\;\inf_{0\leq s \leq t} U_1(s) < 0\,, \;  \inf_{0\leq s \leq t} U_2(s) < 0\right\}\,. 
\eeam
By relations \eqref{eq.KP.3.3} and \eqref{eq.KP.3.4} we understand that $\psi_{\max}$ depicts the probability that both portfolios get simultaneously negative surplus in the interval $[0,\,T]$, while $\psi_{\and}$ depicts the  probability that both portfolios get negative surplus in the interval $[0,\,T]$, but not necessarily simultaneously. Hence we observe that
\beam \label{eq.KP.3.5}
\psi_{\max}(x,\,y;\,T) \leq \psi_{\and}(x,\,y;\,T)\,. 
\eeam
  
The following result represents a partial generalization of \cite[Th. 1.1]{li:2018b}. Although we consider more general dependence structures, as restrict the claim distribution class to $\mathcal{D} \cap \mathcal{L} \subsetneq \mathcal{S}$ and the convergence is as $x\wedge y \to \infty$ instead of $(x,\,y) \to (\infty,\,\infty)$.

We can observe that in this risk model, except the dependence structures appearing in the conditions on the claims, the two business lines are also dependent through the common renewal process. We note that $C>0$ is the common constant of $SAI$ property, for any $(X,\,Y)$, since these random pairs are identically distributed.

\bth \label{th.KP.3.1}
Let consider the bi-variate renewal risk model \eqref{eq.KP.3.1}, with $r \geq 0$. We assume that the $\{(X_i,\,Y_i)\,,\;i\in \bbn\}$, $\{(C_1(t),\,C_2(t))\,,\; t\geq 0\}$ and $\{N(t)\,,\; t\geq 0\}$ are mutually independent and the sequence of pairs $(X_i,\,Y_i)$ satisfy the conditions of Theorem \ref{th.KP.4} under the restriction that $(X_i,\,Y_i)$ are identically distributed random pairs, with marginal distribution $F$ and $G$, respectively. Then, for any finite $T>0$ such that $\lambda(T) > 0$, it holds 
\beao
\psi_{\max}(x,\,y;\,T) &\sim& \psi_{\and}(x,\,y;\,T) \\[2mm] \notag
&\sim& \iint_{s,\,t \geq 0,\,s+t \leq T} \left[\overline{F}(x\,e^{r\,(t+s)})\,\overline{G}(y\,e^{r\,t})+ \overline{F}(x\,e^{r\,t})\,\overline{G}(y\,e^{r\,(t+s)})\right]\,\lambda(ds)\,\lambda(dt) \\[2mm] \notag
&&\qquad \qquad + C\,\int_{0}^T \overline{F}(x\,e^{r\,t})\,\overline{G}(y\,e^{r\,t})\,\lambda(dt)=:\Delta(x,\,y;\,T)\,,
\eeao 
as $x \wedge y \to \infty$.
\ethe

Before presenting the proof of the main result, we need a preliminary lemma, that has it own merit, as it provides the joint tail of the discounted aggregate claims.

\ble \label{lem.KP.3.1}
Under the conditions of Theorem \ref{th.KP.3.1}, then it holds
\beam \label{eq.KP.3.7}
&&\PP[D^{(1)}_r(T)>x\,,\;D^{(2)}_r(T)>y]   \sim \Delta(x,\,y;\,T)\,,
\eeam
as $x\wedge y \to \infty$.
\ele

\pr~
For any $m \in \bbn$ and $x\wedge y \geq 0$, we obtain
\beam \label{eq.KP.3.8} \notag
&&\PP[D^{(1)}_r(T)>x\,,\;D^{(2)}_r(T)>y] = \PP\left[\sum_{i=1}^{N(T)} X_i\,e^{-r\,T_i} >x\,,\;\sum_{j=1}^{N(T)} Y_j\,e^{-r\,T_j} >y \right] \\[2mm]
&&=\left(\sum_{n=1}^{m} + \sum_{n=m+1}^{\infty}  \right)\,\PP\left[\sum_{i=1}^{n} X_i\,e^{-r\,T_i} >x\,,\;\sum_{j=1}^{n} Y_j\,e^{-r\,T_j} >y\,,\;N(T)=n \right]\\[2mm] \notag
&&=:I_1 (x,\,y;\,T) + I_2 (x,\,y;\,T)\,.
\eeam
For the second term, for any $p> J_F^+ \vee J_G^+$, by Markov's inequality, the $SAI$ dependence and the fact that $F,\,G \in \mathcal{D} \cap \mathcal{L} \subsetneq \mathcal{D}$, we can find some constant $L>0$ such that it holds
\beao 
&&I_2 (x,\,y;\,T) \leq \left(\sum_{n=m+1}^{\left\lfloor (x\wedge y)/L\right\rfloor} + \sum_{n=\left\lfloor (x\wedge y)/L\right\rfloor+1}^{\infty}  \right)\PP\left[\sum_{i=1}^{n} X_i>x ,\,\sum_{j=1}^{n} Y_j>y \right]\,\PP[N(T)=n]\\[2mm]
&&\leq \sum_{n=m+1}^{\left\lfloor (x\wedge y)/L\right\rfloor} n^2\,\PP\left[X_i >\dfrac xn\,,\; Y_j > \dfrac yn \right]\,\PP[N(T)=n] +\PP\left[N(T) > \dfrac {x\wedge y}L \right] \\[2mm]
&&\leq \sum_{n=m+1}^{\left\lfloor (x\wedge y)/L\right\rfloor} K\,n^{2(p+1)}\,\overline{F}\left(x\right)\,\overline{G}\left(y\right)\,\PP[N(T)=n] \\[2mm] 
&&\qquad + \left( \dfrac{x\wedge y}L\right)^{-2\,(p+1)}\,\E\left[[N(T)]^{2\,(p+1)} {\bf 1}_{\left\{ N(T) > \dfrac {x\wedge y}L \right\}} \right]\\[2mm] 
&&\lesssim K\,\overline{F}\left(x\right)\,\overline{G}\left(y\right)\,\E\left[[N(T)]^{2\,(p+1)} {\bf 1}_{\{ N(T)>m\}} \right]\,,
\eeao
as $x\wedge y \to \infty$, where the constant $K$ can be found in \cite[Lem. 3(i)]{li:2018}, and in last step we used that, since $F,\,G \in \mathcal{D}$ and $p> J_F^+ \vee J_G^+$, then follows that 
\beam \label{eq.KP.3.9} 
(x\wedge y)^{-2\,p} = o\left[\overline{F}\left(x\right)\,\overline{G}\left(y\right)\right]\,,
\eeam
as $x\wedge y \to \infty$. Indeed, we see that 
\beao
(x\wedge y)^{-p} \leq x^{-p} =o\left[\overline{F}(x)\right]\,, \qquad \qquad (x\wedge y)^{-p} \leq y^{-p}= o\left[\overline{G}(y)\right]\,,
\eeao 
as $x \to \infty$ and $y \to \infty$, respectively, from where we obtain \eqref{eq.KP.3.9}.

We observe that it holds
\beao
\int_0^{T} \PP\left[ X\,e^{-r t} >x,\; Y \,e^{-r t} >y \right]\,\lambda(dt)\sim C\int_0^{T} \bF\left(x\,e^{r t}\right)\,\bG\left(y\,e^{r t}\right)\lambda(dt)\,,
\eeao
because $T>0$ if finite. Furthermore,
\beam \label{eq.KP.3.10} \notag
\int_0^{T} \PP\left[ X\,e^{-r t} >x,\; Y \,e^{-r t} >y \right]\,\lambda(dt)&\geq& \PP\left[ X\,e^{-r T} >x,\; Y\,e^{-r T} >y \right]\,\lambda(T)\\[2mm]
&\asymp& \overline{F}(x)\,\overline{G}(y)\,\lambda(T)\,,
\eeam
as $x\wedge y \to \infty$, where in the last step we used again \cite[Lem. 3(i)]{li:2018}. So, we find out that
\beao
&&\lim_{m\to \infty} \limsup_{x\wedge y \to \infty} \dfrac{I_2(x,\,y;\,T)}{\int_0^{T} \PP\left[ X\,e^{-r t} >x,\; Y \,e^{-r t} >y \right]\,\lambda(dt)}\\[2mm]
&&\qquad \qquad \leq \dfrac {K}{M} \lim_{m\to \infty}\dfrac 1{\lambda(T)}\,\E\left[ [N(T)]^{2\,(p+1)} {\bf 1}_{\{ N(T)>m\}}\right]=0\,,
\eeao
where the constant $M>0$, stems from relation \eqref{eq.KP.3.10}, while in the last step we used \cite[Lem. 3.2]{tang:2007}. Therefore, it holds
\beam \label{eq.KP.3.11} 
\lim_{m\to \infty} \limsup_{x\wedge y \to \infty} \dfrac{I_2(x,\,y;\,T)}{C\,\int_0^{T}  \bF\left(x\,e^{r t}\right)\,\bG\left(y\,e^{r t}\right)\lambda(dt)}=0\,.
\eeam

Now, we estimate the $I_1 (x,\,y;\,T)$. At first, through the dominated convergence theorem, by Theorem \ref{th.KP.4}, because $0 \leq e^{-r\,T_i} \leq 1$, we obtain
\beam \label{eq.KP.3.12} \notag
&&\PP\left[\sum_{i=1}^{n} X_i e^{-r\,T_i} > x,\,\sum_{j=1}^{n} Y_j e^{-r\,T_j} > y ,\,N(T)=n \right]=\int \cdots \int_{\left\{ 
\begin{array}{c}
0\leq t_1 \leq \cdots \leq t_n \leq T\\ 
t_{n+1}>T 
\end{array} 
\right\} }\\[2mm] \notag
&& \times \PP\left[\sum_{i=1}^{n} X_i\,e^{-r\,t_i} > x\,,\;\sum_{j=1}^{n} Y_j\,e^{-r\,t_j} > y \right]\,\PP[T_1 \in dt_1,\,\ldots,\,T_{n+1}\in dt_{n+1}]\\[2mm]
&&\sim\int \cdots \int_{\left\{ 
\begin{array}{c}
0\leq t_1 \leq \cdots \leq t_n \leq T\\ 
t_{n+1}>T  
\end{array} 
\right\} } \left(\sum_{i=1}^{n} \sum_{j=1}^{n} \PP\left[ X_i\,e^{-r\,t_i} > x\,,\; Y_j\,e^{-r\,t_j} > y \right]\right)\\[2mm] \notag
&&\times \PP[T_1 \in dt_1,\,\ldots,\,T_{n+1}\in dt_{n+1}]\sim \sum_{i=1}^{n} \sum_{j=1}^{n} \PP\left[ X_i e^{-r\,T_i} > x,\;Y_j e^{-r\,T_j} > y,\;N(T)=n\right]\,,
\eeam
as $x\wedge y \to \infty$. Therefore, from relations \eqref{eq.KP.3.8} and \eqref{eq.KP.3.12} we conclude
\beam \label{eq.KP.3.13}
&&I_1 (x,\,y;\,T) \sim \sum_{n=1}^{m} \sum_{i=1}^{n} \sum_{j=1}^{n} \PP\left[ X_i\,e^{-r T_i} >x,\; Y_j \,e^{-r T_j} >y\,,\;N(T)=n \right]=\\[2mm] \notag
&&\left(\sum_{n=1}^{\infty} - \sum_{n=m+1}^{\infty}  \right)\,\sum_{i=1}^{n} \sum_{j=1}^{n} \PP\left[ X_i\,e^{-r T_i} >x,\; Y_j \,e^{-r T_j} >y\,,\;N(T)=n \right] =:\sum_{l=3}^4I_l (x,\,y;\,T)\,.
\eeam
Further we follow the line from \cite[Lem. 3.4]{yang:li:2014}, but for convenience we present here the full argument. For the first term we obtain
\beam \label{eq.KP.3.14}
&&I_3 (x,\,y;\,T) = \sum_{i=1}^{\infty} \sum_{n=i}^{\infty} \sum_{j=1}^{n} \PP\left[ X_i\,e^{-r T_i} >x,\; Y_j \,e^{-r T_j} >y\,,\;N(T)=n \right]\\[2mm] \notag
&&=\sum_{i=1}^{\infty} \sum_{n=i}^{\infty} \left(\sum_{j=1}^{i-1} +\sum_{j=i}+ \sum_{j=i+1}^{n}  \right)\,\PP\left[ X_i\,e^{-r T_i} >x,\; Y_j \,e^{-r T_j} >y\,,\;N(T)=n \right] \\[2mm] \notag
&&=\sum_{j=1}^{\infty}\sum_{i=j+1}^{\infty}\,\PP\left[ X_i\,e^{-r T_i} >x,\; Y_j \,e^{-r T_j} >y\,,\;T_i\leq T \right] \\[2mm] \notag
&&+\sum_{i=1}^{\infty}\,\PP\left[ X_i\,e^{-r T_i} >x,\; Y_i \,e^{-r T_i} >y\,,\;T_i\leq T \right] \\[2mm] \notag
&&+\sum_{i=1}^{\infty}\sum_{j=i+1}^{\infty}\,\PP\left[ X_i\,e^{-r T_i} >x,\; Y_j \,e^{-r T_j} >y\,,\;T_j\leq T \right] =:\sum_{k=1}^3 I_{3k} (x,\,y;\,T)\,.
\eeam
For the first term $I_{31} (x,\,y;\,T)$, taking into account that $\{N(t)\,,\; t\geq 0\}$ represents a homogeneous renewal process, then $(T_i - T_j)$ is independent of $T_j$ and furthermore $(T_i - T_j) \stackrel{d}{=} T_{i-j}$, where the equality here, means equal in distribution, hence we find
\beam \label{eq.KP.3.15} \notag
&&I_{31} (x,\,y;\,T) = \sum_{j=1}^{\infty} \sum_{i=j+1}^{\infty}\,\PP\left[ X_i\,e^{-r (T_j +(T_i-T_j))} >x,\; Y_j \,e^{-r T_j} >y\,,\;T_j+(T_i - T_j)\leq T \right] \\[2mm] \notag
&&=\sum_{j=1}^{\infty} \sum_{i=j+1}^{\infty}\,\iint_{s,t\geq 0, s+t\leq T}\PP\left[ X_i\,e^{-r (t +s)} >x,\; Y_j \,e^{-r t} >y\right]\,\PP[T_{i-j} \in ds]\,\PP[T_j \in dt] \\[2mm]
&&=\iint_{s,t\geq 0, s+t\leq T}\overline{F}(x\,e^{r (t +s)})\,\overline{G}(y\,e^{r \,t}) \,\lambda(ds)\,\lambda(dt) \,.
\eeam
By symmetry, we obtain
\beam \label{eq.KP.3.16} \notag
&&I_{33} (x,\,y;\,T) = \sum_{i=1}^{\infty} \sum_{j=i+1}^{\infty}\,\PP\left[ X_i\,e^{-r \,T_i} >x,\; Y_j \,e^{-r (T_i+(T_j-T_i))} >y\,,\;T_i+(T_j - T_i)\leq T \right] \\[2mm]
&&=\iint_{s,t\geq 0, s+t\leq T}\overline{F}(x\,e^{r\,t})\,\overline{G}(ye^{r (t +s)})\, \,\lambda(ds)\,\lambda(dt) \,.
\eeam
Finally, by the $SAI$ dependence between $X_i$ and $Y_i$, we have
\beam \label{eq.KP.3.17} \notag
I_{32} (x,\,y;\,T) &=& \sum_{i=1}^{\infty} \int_0^T \PP\left[ X_i\,e^{-r \,t} >x,\; Y_i \,e^{-r t} >y \right]\,\PP[T_i \in dt] \\[2mm]
&\sim& C\,\int_{0}^T \overline{F}(x\,e^{r \,t})\,\overline{G}(ye^{r t})\,\lambda(dt)=:C\,I(x,\,y,\,F,\,G,\,T) \,,
\eeam
as $x\wedge y \to \infty$. So, by relations \eqref{eq.KP.3.14} - \eqref{eq.KP.3.17} we find
\beam \label{eq.KP.3.18}
I_{3} (x,\,y;\,T) \sim \Delta(x,\,y;\,T)\,,
\eeam
as $x\wedge y \to \infty$.

Now, we estimate $I_{4} (x,\,y;\,T)$. It is easy to see that for sufficiently large $x \wedge y$, we get
\beao
&&I_4 (x,\,y;\,T) = \sum_{n=m+1}^{\infty} \sum_{i=1}^{n} \sum_{j=1}^{n} \PP\left[ X_i\,e^{-r\,T_i} >x,\; Y_j \,e^{-r\,T_j} >y\,,\;N(T)=n \right]\\[2mm] \notag
&&\leq \sum_{n=m+1}^{\infty} \sum_{i=1}^{n}\left( \sum_{i \neq j=1}^{n} +  \sum_{i=j=1}^n \right)\,\PP\left[ X_i\,e^{-r\,T_1} >x,\; Y_j \,e^{-r T_1} >y\,,\;T_n\leq T \right] = \sum_{n=m+1}^{\infty} \\[2mm] \notag
&&\sum_{i=1}^{n} \left(\sum_{i\neq j=1}^{n} + \sum_{i=j=1}^n \right)\,\int_0^T \PP\left[ X_i\,e^{-r\,t} >x,\; Y_j \,e^{-r\,t} >y\right] \,\PP[N(T-t) \geq n-1]\,\PP[T_1 \in dt]\\[2mm] \notag
&&\leq K\,\sum_{n=m+1}^{\infty}[n\,(n-1) +2\,n]\int_{0}^{T}\overline{F}(x\,e^{r \,t})\,\overline{G}(y\,e^{r\,t}) \,\PP[N(T-t) \geq n-1]\,\PP[T_1 \in dt] 
\\[2mm] \notag
&&\leq K\,\sum_{n=m+1}^{\infty}(n^2+n)\,\PP[N(T) \geq n-1]\,\int_{0}^{T}\overline{F}(x\,e^{r \,t})\,\overline{G}(y\,e^{r\,t}) \,\PP[T_1 \in dt]\,.
\eeao
where the constant $K>0$, comes from \cite[Lem. 3(i)]{li:2018}. Hence, for any large enough $m$ and $x\wedge y> 0$, we can find some small enough $\varepsilon>0$, such that $I_4 (x,\,y;\,T)\leq \varepsilon\,K\,I(x,\,y,\,F,\,G,\,T) \leq \varepsilon\,K\,I_3 (x,\,y;\,T)$ and letting $m$ tends  to infinity, we can take $\varepsilon \downarrow 0$. Thus by relations \eqref{eq.KP.3.13}, \eqref{eq.KP.3.18} and last asymptotic inequality we find 
\beam \label{eq.KP.3.20}
I_{1} (x,\,y;\,T) \sim I_{3} (x,\,y;\,T) \sim \Delta(x,\,y;\,T)\,,
\eeam
as $x\wedge y \to \infty$. Therefore, by  \eqref{eq.KP.3.8}, \eqref{eq.KP.3.11} and  \eqref{eq.KP.3.20}, we reach the relation  \eqref{eq.KP.3.7}.
~\halmos

{\bf Proof of Theorem \ref{th.KP.3.1}.}~
For the $\psi_{\and}$, by Lemma \ref{lem.KP.3.1} we find
\beam \label{eq.KP.3.21} \notag
\psi_{\and} (x,\,y;\,T) &=&\PP\left[\inf_{0<t\leq T}U^{(1)}(t) < 0\,, \inf_{0<t\leq T}U^{(2)}(t) < 0\right]\\[2mm]
&\leq& \PP\left[D_r^{(1)}(T) >x\,, \;D_r^{(2)}(T) >y\right] \sim  \Delta(x,\,y;\,T) \,,
\eeam
as $x \wedge y \to \infty$. From the other side, for the lower bound of $\psi_{\max}$, following the line from \cite[Th. 2.1]{yang:li:2014}, we obtain
\beao
\psi_{\max}(x,\,y;\,T) &=&\PP\left[\inf_{0<t\leq T} \{U^{(1)}(t) \vee U^{(2)}(t) \}< 0\right]\\[2mm]
&\geq& \PP\left[D_r^{(1)}(T) -\int_{0}^{T}e^{-r \,s}\,C_1(ds)> x\,,\;D_r^{(2)}(T) -\int_{0}^{T}e^{-r \,s}\,C_2(ds)> y\right] \\[2mm]
&=& \int_0^{\infty} \int_0^{\infty} \PP\left[ D_r^{(1)}(T) >x+u\,,\;D_r^{(2)}(T) >y+z\right]\,H(du,\,dz)\\[2mm]
&\sim& \int_0^{\infty} \int_0^{\infty} \Delta(x+u,\,y+z\,;\,T)\,H(du,\,dz)\,,
\eeao
as $x \wedge y \to \infty$, where by $H$ we denote the distribution of the 
\beao
\left(\int_{0}^{T}e^{-r \,s}\,C_1(ds)\,,\; \int_{0}^{T}e^{-r \,s}\,C_2(ds) \right)\,,
\eeao
and in last step we used Lemma \ref{lem.KP.3.1}. Therefore, due to the fact that $F,\,G \in \mathcal{D}\cap \mathcal{L} \subsetneq \mathcal{L}$, we obtain 
$\Delta(x+u,\,y+z\;\,T) \sim \Delta(x,\,y\;\,T)$, as $x \wedge y \to \infty$, which in combination of the last relation, by dominated convergence theorem we find
\beam \label{eq.KP.3.22} 
\psi_{\max} (x,\,y;\,T) \gtrsim \Delta(x,\,y;\,T) \,,
\eeam 
as $x \wedge y \to \infty$. From relations  \eqref{eq.KP.3.21}, \eqref{eq.KP.3.22} and \eqref{eq.KP.3.5} we get the desired result.
~\halmos

\section{Tail distortion risk measures} \label{sec.KP.5}

Let us remind the Hadamard product of two non-negative, $n$-variate random vectors ${\bf \Theta}$ and ${\bf X}$. Here, the vector ${\bf \Theta}=(\Theta_1\,\,\ldots,\,\Theta_n)$ has non-negative and non-degenerated to zero components and represents the systemic risk factors, while the non-negative random vector ${\bf X}=(X_1,\,\ldots,\,X_n)$ describe the losses of $n$ portfolios, namely, the random variable $X_i$ represents the loss of the $i$-th portfolio, over a concrete time horizon, with $i=1,\,\ldots,\,n$. Hence, the product ${\bf \Theta}\,{\bf X}$ corresponds to discounted claim of the $n$ portfolios over a concrete time horizon.

In order to make the model more realistic, we allocate the initial capital into $n$ lines of business, in general of different amounts. Thus we need non-random, positive weights $w_1,\,\ldots,\,w_n$, with $\sum_{i=1}^n w_i =1$. Then the discounted aggregate loss of portfolio is presented as
\beao
{\bf \Theta\,X(w)}:=\sum_{i=1}^n w_i\,\Theta_i \,X_i\,.
\eeao

This model called background risk model and it can describe the systemic risk. For more details about Background risk model see \cite{asimit:vernic:zitikis}, \cite{cote:genest} and \cite{tsanakas}. We intent to study a risk measure, called tail distortion risk measure, symbolically $TDRM$, in the background risk model ${\bf \Theta}\,{\bf X}({\bf w})$, under a condition, see Assumption \ref{ass.KP.4.1} below, which permits several forms of dependence, among the systemic risk factors ${\bf \Theta}$, among the losses ${\bf X}$ and between the ${\bf \Theta}$ and ${\bf X}$ simultaneously. Before the definitions we give a survey of the current literature on the topic. 

 In \cite{chen:wang:zhang:2022} was examined the distortion tail risk measure in a similar background risk model. In \cite{chen:wang:zhang:2022} are developed asymptotic results for the tail distortion risk measure $TDRM$ of the quantity $ \Theta\,{\bf X(w)}$ and the weights $ w_{i}$ for $i=1,\,\ldots,\,n$, defined as above, with $\Theta$ to be one-dimensional, and therefore is a common systemic risk factor for all business lines, and independent of ${\bf X}$. Also they assume left-continuous distortion function. In \cite{yang:liu:liu:2023} was studied the $TDRM$ for right-continuous distortion function in the $X_n(\Theta)$  model, where was permitted multivariate systemic risk factors. In this work was used either $MRV$ structure for the vector ${\bf X}$ or some general enough structures of asymptotic independence for its components (with regularly varying tails), which can provide a direct asymptotic expression. However, the vectors  ${\bf \Theta}$ and ${\bf X}$ are still independent. 

In this work we study the ${\bf \Theta}\,{\bf X}({\bf w})$ model, in which through Assumption \ref{ass.KP.4.1} is permitted arbitrary dependence among each vectors components and also is permitted dependence between ${\bf \Theta}$ and ${\bf X}$, that represents a mathematical generalization with obvious practical impact.

\begin{assumption} \label{ass.KP.4.1}
Let ${\bf \Theta}\,{\bf X} \in MRV(\alpha,\,V,\,\mu)$, for some $\alpha \in (0,\,\infty)$.
\end{assumption}

\bre \label{rem.KP.4.1}
Assumption \ref{ass.KP.4.1} is satisfied in many cases, where ${\bf X} \in MRV(\alpha,\,V,\,\mu^*)$, for some Radon measure $\mu^*$, through multivariate versions of Breiman's theorem. For example, under some moment conditions for the components of ${\bf \Theta}$, in \cite{basrak:davis:mikosch:2002} we find that in case of independent  ${\bf \Theta}$ and ${\bf X}$, the $MRV$ structure remains in the product ${\bf \Theta}\,{\bf X}$, with a new Random measure, but with same regular variation index and same normalization function. Later, in \cite{hult:samorodnitsky:2008} we meet a similar result, with respect to a predictable framework of ${\bf \Theta}$ and ${\bf X}$ and next in \cite{fougeres:mercadier:2012}, we have an extension of the result in dependent ${\bf \Theta}$ and ${\bf X}$. For the special case, where ${\bf \Theta}=\Theta\,{\bf 1}$, under a rather weak dependence structure between $\Theta$ and ${\bf X}$, suggested in \cite{li:2016}, we can find in \cite{cheng:konstantinides:wang:2024b} that the product $\Theta\,{\bf X}$ keeps the $MRV$ structure, with same regular variation index and same normalizing function. 
\ere

\bre \label{rem.KP.4.2}
An important property for the next proofs, is the closure one of $MRV$ class, with respect to linear combinations. Namely, if we have a non-negative random vector ${\bf X} \in MRV(\alpha,\,V,\,\mu^*)$, then the distribution of all the non-negative, non-degenerated to zero linear combinations $\sum_{i=1}^n l_i\,X_i$, belongs to class $\mathcal{R}_{-\alpha}$, with same index $\alpha$, see for example in \cite[Sec. 7.3.1]{resnick:2007}. Hence, under Assumption \ref{ass.KP.4.1} we obtain that ${\bf \Theta}\,{\bf X}({\bf w}) \in \mathcal{R}_{-\alpha}$.
\ere

Now in the classic one-dimensional Background risk model we study the asymptotic behavior of tail distortion risk measure which is more general than conditional tail expectation. The following class of measures was introduced in \cite{wang:1996}.
For a given non-decreasing function $g:[0,1] \rightarrow [0,1]$, such that $g(0)=0$, $g(1)=1$, and for any non-negative random variable $X$, with distribution $F$, the measure
\beao 
	\rho_g[X]=\int_{0}^{\infty}g\left[\bF(x)\right]\,dx\,,
\eeao
is called distorted risk measure and the function $g$ is called distortion function.
It is well-known that the $VaR$ and $CTE$ are distortion risk measures
\beao
VaR_p(Z):=\inf\{x\in \mathbb{R}\;:\;V(x)\geq p\}\,, \qquad CTE_p(Z) = \E[Z\;|\;Z>VaR_p(Z)]\,,
\eeao
for some $p \in (0,\,1)$, see for example in \cite{zhu:li:2012}. In last reference, in order to stress on the risk tail, was introduced by \cite{zhu:li:2012} the definition of $TDRM$ as follows.

\begin{definition} \label{def.KP.1.3}
Let $g\;:\;[0,1]\rightarrow[0,1]$ be non-decreasing, such that $g(0)=0$ and $g(1)=1$, then the tail distortion risk measure of a non-negative random variable $X$ with distribution $F$ is given by
\beao
	\rho_{g}[X\;|\;X>VaR_{p}(X)]=\int_{0}^{\infty}g\left[\bF_{X|X>VaR_{p}(X)}(x)\right]\,dx\,,
\eeao
where $\bF_{X|X>y}(x)=\PP(X>x\;|\;X>y)$. 
\end{definition} 

It is well-known that, the tail distortion risk measure of a continuous random variable $X$, is a distortion risk measure. Furthermore, if the distortion function is the identical function $g(x)=x$, then the  tail distortion risk measure coincides with the conditional tail expectation. 

From \cite{zhu:li:2012} we find that if $X\in \mathcal{R}_{-\alpha}$, with $\alpha>0$ and additionally
\beao
\int_{1}^{\infty}g\left( \dfrac 1{y^{\alpha-k}} \right)\,dy<\infty\,, 
\eeao
for some $0<k<\alpha$, then $\rho_{g}[X|X>VaR_{p}(X)] \sim C_{\alpha}(g)\,VaR_{p}(X)$, as $p\rightarrow 1$, where 
\beao
C_{\alpha}(g):=\int_{0}^{1}y^{-1/\alpha}\,g(dy)=1+\int_{1}^{\infty}g\left(\dfrac 1{y^{\alpha}} \right)\,dy\,. 
\eeao
This last result was shown for any distortion function, without continuity requirement. 

From \cite{dhaene:vanduffel:goovaerts:kaas:tang:2006} we find for the function $B_{X}$, defined as quantile of $1/\bF$,
\beam \label{eq.KP.1}
B_{X}(s):=\left(\dfrac{1}{\bF}\right)^{\leftarrow}(s)=F^{\leftarrow}\left(1-\dfrac{1}{s}\right)\,,
\eeam 
for any $s\geq 1$, and for some right-continuous distortion function $g$, holds the representation
\beam \label{eq.KP.2}
\rho_{g}[X\;|\;X>VaR_{p}(X)]=\int_{0}^{1} B_{X}\,\left(\dfrac{1}{y(1-p)}\right)\,g(dy)\,,
\eeam
for any non-negative random variable $X$ with distribution $F$.

It is clear that the previous asymptotic expressions in section \ref{sec.KP.5} depends on the distortion function only though the constant $C_{\alpha}(g)$, and therefore the quantiles of random variables remain clear of distortion. This helps in practical applications, since the only we need for the application of the model, when the distortion function is varying for the same risks, is the  calculation of the integral. 

Next, we study the background risk under Assumption \ref{ass.KP.4.1}. So, in this model, we allow dependence among the losses of the $n$ lines of business and dependence among the systemic risk factors with the losses.

\bth \label{th.KP.4.1}
Let the product ${\bf \Theta}\,{\bf X}$ satisfy Assumption \ref{ass.KP.4.1} and condition 
\beam \label{eq.KP.11,5}
\int_{1}^{\infty}g\left(\dfrac 1{y^{\alpha/(1+\zeta)}}\right)dy<\infty\,,
\eeam 
for some $\zeta>0$. Then it holds
\beam  \label{eq.KP.12}
\rho_{g}\left[ {\bf \Theta\,X(w)}\;\big|\; {\bf \Theta\,X(w)}> VaR_{p}\left({\bf \Theta\,X(w)} \right)\right] \sim C_{\alpha}(g)\,\dfrac{\gamma_{\bf w}^{1/\alpha} }{ \Gamma_\alpha}\,\sum_{i=1}^{n} VaR_{p}\left(\Theta_{i}X_{i}\right)\,,
\eeam
as $p\rightarrow 1$, with
\beam \label{eq.KP.13,5}
\gamma_{\bf w}:=\lim_{\xto} \dfrac{\PP({\bf \Theta\,X(w)}  >x )}{\PP\left[\sum_{i=1}^{n} \Theta_{i}X_{i} >x\right]}\,,\qquad \Gamma_\alpha:=\sum_{i=1}^{n}\gamma_{e_{i}}^{1/\alpha}\,.
\eeam
\ethe

\begin{proof}
From Assumption \ref{ass.KP.4.1} we obtain ${\bf \Theta\,X}\in MRV(\alpha,\,V,\,\mu) $, therefore ${\bf \Theta\,X(w)} \in \mathcal{R}_{-\alpha}$. From relation \eqref{eq.KP.2} and the fact that
\beam \label{eq.KP.13.7}
VaR_{p}\left( {\bf \Theta\,X(w)}\right)=B_{{\bf \Theta\,X(w)}}\left(\dfrac{1}{1-p}\right)\,,
\eeam
see \eqref{eq.KP.1}, we can show
\beam \label{eq.KP.13}
\dfrac{1}{ B_{{\bf \Theta\,X(w)}}\left(\dfrac{1}{1-p}\right)} \int_{0}^{1} B_{{\bf \Theta\,X(w)}}\left(\dfrac{1}{y(1-p)}\right)\,g(dy) \sim C_{\alpha}(g)\,,
\eeam
as $p\rightarrow 1$. Indeed, from the \cite{drees:1998}, \cite[Th. \,B.2.18]{dehaan:ferreira:2006}  we get that there exists a  $0< \tilde{p} <1$, which depends on $\zeta >0$, such that for $ \tilde{p} \leq p <1 $ and $ 0< y <1 $ holds
\beam \label{eq.KP.14} \notag
	&&\left| \dfrac{B_{{\bf \Theta\,X(w)}}\left(\dfrac{1}{y(1-p)}\right)-B_{{\bf \Theta\,X(w)}}\,\left(\dfrac{1}{1-p}\right)}{\dfrac 1{\alpha}\,B_{{\bf \Theta\,X(w)}}\,\left(\dfrac{1}{1-p}\right)}-\dfrac{y^{-1/\alpha}-1}{1/\alpha} \right|\\[2mm]  
	&&\qquad \qquad = \left| \dfrac{B_{{\bf \Theta\,X(w)}}\left(\dfrac{1}{y(1-p)}\right) }{\dfrac 1{\alpha}\,B_{{\bf \Theta\,X(w)}}\left(\dfrac{1}{1-p}\right)} - \dfrac{y^{-1/\alpha}}{1/\alpha} \right| \leq y^{-(1+\zeta)/\alpha}\,.
 \eeam
Further we have
\beam \label{eq.KP.15}
	\left| \dfrac{B_{{\bf \Theta\,X(w)}}\left(\dfrac{1}{y(1-p)}\right)}{\dfrac 1{\alpha}\,B_{{\bf \Theta\,X(w)}}\left(\dfrac{1}{1-p}\right)}\right| -\dfrac{y^{-1/\alpha}}{1/\alpha} \leq \left|\dfrac{B_{{\bf \Theta\,X(w)}}\left(\dfrac{1}{y(1-p)}\right)}{\dfrac 1{\alpha}\,B_{{\bf \Theta\,X(w)}}\left(\dfrac{1}{1-p}\right)}-\dfrac{y^{-1/\alpha}}{1/\alpha} \right|\,,
\eeam
Hence from relations \eqref{eq.KP.14} and  \eqref{eq.KP.15} we obtain
\beam \label{eq.KP.16}
\left| \dfrac{B_{{\bf \Theta\,X(w)}}\left(\dfrac{1}{y\,(1-p)}\right)}{B_{{\bf \Theta\,X(w)}}\left(\dfrac{1}{1-p}\right)} \right| \leq y^{-1/\alpha}+\dfrac{y^{-(1+\zeta)/\alpha}}{\alpha}\,.
\eeam
Since the integral in \eqref{eq.KP.11,5} converges for some $\zeta>0$, it follows 
\beao
\int_{0}^{1} y^{-1/\alpha}\,g(dy) \leq \int_{0}^{1} y^{-(1+\zeta)/\alpha}\,g(dy)< \infty \,,
\eeao 
whence we obtain
\beao
\int_{0}^{1} \left( y^{-1/\alpha}+\dfrac{y^{-(1+\zeta)/\alpha}}{\alpha}
\right)\,g(dy)<\infty\,.
\eeao
Therefore from \eqref{eq.KP.16} and by dominated convergence theorem, we find
\beao
\lim_{p\rightarrow 1}\int_{0}^{1} \dfrac{B_{{\bf \Theta\,X(w)}}\left(\dfrac{1}{y\,(1-p)}\right)}{B_{{\bf \Theta\,X(w)}}\left(\dfrac{1}{1-p}\right)}\,g(dy)&=& \int_{0}^{1}\lim_{p\rightarrow 1} \dfrac{B_{{\bf \Theta\,X(w)}}\,\left(\dfrac{1}{y\,(1-p)}\right)}{B_{{\bf \Theta\,X(w)}}\left(\dfrac{1}{1-p}\right)}\,g(dy)\\[2mm]
	&=&\int_{0}^{1}y^{-1/\alpha}\,g(dy)=C_{\alpha}(g) \,,
\eeao
where the pre-last equality follows from $B_{{\bf \Theta\,X(w)}}(\cdot) \in \mathcal{R}_{1/\alpha}$, see \eqref{eq.KP.1} and \cite[sec. 2.4]{resnick:2007}. Therefore, relation \eqref{eq.KP.13} is true. 

Furthermore, taking into account relation \eqref{eq.KP.13,5} we obtain 
\beam \label{eq.KP.17}
\lim_{p\uparrow 1}\dfrac{\sum_{i=1}^{n} VaR_{p}(\Theta_{i}X_{i})}{VaR_{p}\left({\bf \Theta\,X(w)}\right)} =\lim_{\xto}\left(\dfrac{\sum_{i=1}^{n} \PP[\Theta_{i}\,X_{i}>x]}{\PP\left[\sum_{i=1}^{n} w_i\,\Theta_{i}\,X_{i}>x\right]}\right)^{1/\alpha}= \dfrac{\Gamma_\alpha}{\gamma_{\bf w}^{1/\alpha}}\,.
\eeam
However, we also obtain
\beao 
\rho_{g}\left[{\bf \Theta\,X(w)} \;\big|\; {\bf \Theta\,X(w)}> VaR_{p}\left({\bf \Theta\,X(w)} \right) \right] = \int_{0}^{1}B_{{\bf \Theta\,X(w)}}\left(\dfrac{1}{y\,(1-p)}\right)\,g(dy)\,,
\eeao
while by \eqref{eq.KP.13.7} and \eqref{eq.KP.13} we find
\beao
\int_{0}^{1} B_{{\bf \Theta\,X(w)}}\left(\dfrac{1}{y\,(1-p)}\right)\,g(dy) \sim C_{\alpha}(g)\,VaR_{p}\left({\bf \Theta\,X(w)}\right)\,,
\eeao
as $p\rightarrow 1 $ and consequently from relation \eqref{eq.KP.17} and the last two expressions we find relation \eqref{eq.KP.12}.
\end{proof}

Next, we provide a simple corollary, which serves as an extension of the result from \cite[Theorem 3.1]{chen:wang:zhang:2022}, in case when the systemic risk factor is independent of the risk vector ${\bf X} \in MRV$.

\begin{corollary} \label{cor.KP.4.3}
Let ${\bf X} \in MRV(\alpha,\,V,\,\mu^*)$ and ${\bf \Theta}$ be independent of ${\bf X}$. We assume that there exists some $\vep>0$, such that $\E\left(\Theta_{i}^{\alpha + \vep}\right)<\infty$, for any $i=1,\,\ldots,\,n$. If the integral in \eqref{eq.KP.11,5} converges, then it holds 
\beam \label{eq.KP.18}
\rho_{g} \left[{\bf \Theta\,X(w)} \;\big|\; {\bf \Theta\,X(w)} > VaR_{p}\left({\bf \Theta\,X(w)} \right) \right]  \sim C_{\alpha}(g)\,\dfrac{\gamma_{\bf w}^{1/\alpha}}{\Gamma_\alpha} \sum_{i=1}^{n} \left[\E\left(\Theta_{i}^{\alpha}\right)\right]^{1/\alpha}\,VaR_{p}(X_{i}) \,,
\eeam
as  $p\rightarrow 1$.
\end{corollary}
	
\begin{proof}
Firstly, we have ${\bf \Theta}\,{\bf X}\in MRV(\alpha,\,V,\,\mu) $ from \cite{basrak:davis:mikosch:2002}. Hence, by Theorem \ref{th.KP.4.1} we obtain \eqref{eq.KP.12}. Because $\Theta_{i}$ and $X_{i}$ are independent, from \cite{chen:gao:gao:zhang:2018} we find 
\beam \label{eq.KP.19}
VaR_{p}(\Theta_{i}\,X_{i})\sim \left[\E\left(\Theta_{i}^{\alpha}\right) \right]^{1/\alpha} \,VaR_{p}(X_{i})\,,
\eeam
as $p\rightarrow 1$. Hence from relation \eqref{eq.KP.19} and relation \eqref{eq.KP.12} we obtain \eqref{eq.KP.18}.
\end{proof}

\section{Conclusion} \label{sec.KP.6}

In the corresponding literature, we find a confusion with respect to term multivariate heavy tails. The reason for this, is that the multivariate distribution tail can not be determined uniquely. Having in mind the single big jump principle, the $d$-dimensional, non-weighted form of \eqref{eq.KP.44} is focused on the following joint distribution tail
\beam \label{eq.KP.6.1}
\PP\left[ \sum_{i=1}^{n_1} X_i^{(1)} >x_1,\,\ldots,\, \sum_{i=1}^{n_d} X_i^{(d)} >x_d \right]\sim \sum_{i=1}^{n_1} \cdots \sum_{i=1}^{n_d}\PP\left[  X_i^{(1)} >x_1,\,\ldots,\,  X_i^{(d)} >x_d \right]\,,
\eeam
as $\wedge_{i=1}^d x_i\to \infty$. Along with approach of the multivariate subexponentiality by \cite{samorodnitsky:sun:2016}, we find in \cite{konstantinides:passalidis:2024g}, under different conditions the relation
\beam \label{eq.KP.6.2}
\PP\left[ \sum_{i=1}^{n} {\bf X}^{(i)} \in x\,A \right]\sim \sum_{i=1}^{n} \PP\left[ {\bf X}^{(i)} \in x\,A \right]\,,
\eeam
as $\xto$, where ${\bf X}^{(1)},\,\ldots,\,{\bf X}^{(n)}$ represent $d$-dimensional random vectors, and $A$ is some rare set. Although the single big jump approximation in \eqref{eq.KP.6.2} has several good properties, as for example that most of the closure properties of the uni-variate distribution classes hold also for the multivariate ones and further some results can be generalized to multidimensional set up, still there exist some basic drawbacks. 

The main drawback is that the set $A$ does not represent the joint distribution tail, see in \cite[Rem. 2.2]{konstantinides:passalidis:2024g}. This leads to a 'linear' approximation of the single big jump, in the sense that in order to happen this representation, should be the from $1$ up to $d$, from the total $n\times d$, random variables, sufficiently 'large'. As consequence, the approximation of relation \eqref{eq.KP.6.2} can not represent the probabilities $\psi_{\max}$ and $\psi_{\and}$, belonging among the four most popular ruin probabilities in multivariate risk models.

In opposite direction, relation \eqref{eq.KP.6.1}, studied in section 2 with $d=2$, can solve the problem with these two ruin probabilities, since it is focused on the joint distribution tail. Furthermore, we can say that relation \eqref{eq.KP.6.1}, follows a 'non-linear' approach of the single big jump, since requires $d$ sufficiently 'large' random variables, namely one big jump for each line of business, if you look from the risk theory aspect. In this sense, the two approached by \eqref{eq.KP.6.1} and \eqref{eq.KP.6.2}, work in complementary modes.    

Finally, we want to mention that the $MRV$, although can give solutions with respect to '$\psi_{\max}$' and '$\psi_{and}$', satisfies only the linear single big jump, namely relation \eqref{eq.KP.6.2}. From this point of view the study of \eqref{eq.KP.6.1} leads to revision of multivariate heavy-tailed distributions, even of the well-established ones, like the $MRV$. The events described by relation \eqref{eq.KP.6.2} do NOT put emphasis in the dimension, in opposite with the events described by relation \eqref{eq.KP.6.1}. 

\vspace{2mm}
\noindent \textbf{Acknowledgments.} 
We feel the need to express our gratitude to two anonymous referees, for their crucial remarks, that improved the present paper.

\end{document}